\documentclass[11pt]{article}
\usepackage{amsmath,blkarray,amssymb,amsthm,graphicx,booktabs,multirow,rotating,tabularx}
\hfuzz=\maxdimen
\newtheorem{thm}{Theorem}[section]

\newtheorem{lem}[thm]{Lemma}

\newtheorem{cor}[thm]{Corollary}
\newtheorem{cla}[thm]{Claim}

\theoremstyle{definition}
\newtheorem{exam$s$-union ple}[thm]{Example}

\newtheorem{rmk}[thm]{Remark}

\newcommand{\GL}{{\operatorname{GL}}}
\newcommand{\GF}{{\operatorname{GF}}}

\begin{document}
	\date{}
	\title{New infinite families of $q$-analogs of group divisible designs with arbitrary block dimension\footnote{Supported by  National Natural Science Foundation of China under Grant 12571346.}}	
	\author{
		{\small Yakun  Wu,}  {\small Junling  Zhou,} {\small Xiaoran  Wang}\\
		{\small School of Mathematics and Statistics}\\ {\small Beijing Jiaotong University}\\
		{\small Beijing  100044, China}\\
		{\small jlzhou@bjtu.edu.cn}\\
	}
	\maketitle
	
	\begin{abstract}
		This paper is mainly devoted to constructions of \(q\)-analogs of group divisible designs and their applications. We give a complete description of the action of \(G=\GL(m,q^l)\) on \(\Omega_k^{k-1}\), where $3\leq k\leq \min\left\lbrace m+1,l\right\rbrace $ and \(\Omega_k^{k-1}\) is the set of \(k\)-subspaces of $\GF(q)^{ml}$ whose \(\GF(q^l)\)-span has dimension \(k-1\). We do this by relating the \(G\)-orbits on \(\Omega_k^{k-1}\) to the corresponding Singer cycle orbits on subspaces of $\GF(q)^l$. From the properties of the $G$-incidence matrix between $2$-subspaces and $k$-subspaces, we obtain plenty of new infinite families of simple \(q\)-analogs of group divisible designs with arbitrary block dimension. We further establish a recursive construction for simple \(q\)-analogs of pairwise balanced designs and then produce new infinite families of such designs. We also obtain plenty of infinite families of non-simple subspace \(2\)-designs through the above two types of designs.

		\medskip
		
		\noindent {\bf Keywords}: $q$-Analog of a group divisible design \ \ Subspace design  \ \ \(q\)-Analog of a pairwise balanced design \ \ Kramer--Mesner method
		
		\medskip
		
		\noindent{\bf 2010  MR Subject Classification}: 05B05, 51E05, 51E20
	\end{abstract}
	
	\section{Introduction}
	Throughout this paper, let $q$ be a prime power and let $V$ be a vector space of dimension $v$ over $\GF(q)$. For brevity, we speak of an $i$-dimensional subspace as an $i$-subspace. A \emph{$q$-analog of a group divisible design} (briefly, a \emph{$q$-GDD}), denoted by a $(v,g,k,\lambda)_q$-GDD, is a triple $(V,\mathcal{G},\mathcal{B})$, where $\mathcal{G}$ is a partition of the set of $1$-subspaces of $V$ into $g$-subspaces, called \emph{groups}, and $\mathcal{B}$ is a collection of $k$-subspaces of $V$, called \emph{blocks}, satisfying the following conditions:
	\begin{enumerate}
		\item[\rm(i)] $\dim_{\GF(q)}(B \cap G) \le 1$ for any $G \in \mathcal{G}$ and $B \in \mathcal{B}$;
		\item[\rm(ii)] Every $2$-subspace $T \leq V$ that is not contained in any group is contained in exactly $\lambda$ blocks of $\mathcal{B}$.
	\end{enumerate}
	According to a statement of Tits \cite{JT}, the combinatorics of sets can be regarded as the limiting case \(q \to 1\) of the combinatorics of vector spaces over \(\GF(q)\). Thus, a \(q\)-GDD can be regarded as a natural extension of a classic group divisible design to vector space over finite fields. Buratti et al. \cite{buratti} conducted a systematic study of \(q\)-GDDs, established necessary conditions for their existence, and clarified their connections with scattered subspaces, \(q\)-Steiner systems, packing designs, and \(q^r\)-divisible projective sets. They also constructed several infinite families of \(q\)-GDDs. A design is said to be \emph{simple} if the block set does not contain repeated blocks and non-simple otherwise.

	\begin{thm}\rm{\cite{Schram}}
		Let $v$ be odd and let $m$ be a divisor of $v$. Then there exists a
		$(v,m,3,q^{2}+q+1)_q$-GDD. Moreover, this $q$-GDD can be
		simple when $3 \mid m$.
	\end{thm}

	\begin{thm}\rm{\cite{buratti,WZ}}\label{kkkk}
		Let $l,m,k$ be integers with $l \ge 2$ and
		$3 \le k \le m$. Then there exists a simple $(ml,l,k,\lambda)_q$-GDD with
		\[
		\lambda
		= q^{(l-1)\left(\binom{k}{2}-1\right)}
		\prod_{i=2}^{k-1}\frac{q^{(m-i)l}-1}{q^{\,k-i}-1}.
		\]
	\end{thm}
	
	In complete analogy with the classic situation, $q$-GDDs are closely related to $q$-analogs of $2$-designs and $q$-analogs of pairwise balanced designs. Firstly, we recall the definition of a $q$-analog of a $t$-design. We also explain its relationship with $q$-GDD and present several important examples.
	
	 A \emph{$q$-analog of a $t$-design}  with parameters $t$-$(v,k,\lambda)_q$ is a pair $\left( V,\mathcal{B}\right) $, where $\mathcal{B}$ is a collection of $k$-subspaces of $V$, called blocks, with the property that each $t$-subspace of $V$ is contained in exactly $\lambda$ blocks of $\mathcal{B}$. A $q$-analog of a $t$-design is also called a \emph{subspace $t$-design} or a \emph{$t$-design over the finite field}. These objects are natural $q$-analogs of the classic $t$-$(v,k,\lambda)$ designs. Because of their applications to error correction in network coding, the study of $t$-$(v,k,\lambda)_q$ designs has attracted considerable attention and has been extensively developed, see \cite{error,koetter}. A $t$-$(v,k,\lambda)_q$ design is said to be \emph{trivial} if its block set consists of all $k$-subspaces, each with the same multiplicity.
	
	  The first known nontrivial infinite family of subspace $2$-designs was obtained by Thomas in \cite{Thomas}. Subsequently, Suzuki \cite{Suzuki1,Suzuki2} and Schram \cite{Schram} further generalized this result from the binary field to general finite fields. Since then, the theory of subspace designs has developed considerably; see~\cite{steiner,large set,difference,Itoh,derived, Miyakawa,Wz,WZ}. In particular, a $t$-$(v,k,1)_q$ design with $t\geq 2$ is called a \emph{$q$-Steiner system} and is denoted by $\mathcal{S}_q(t,k,v)$. An important nontrivial example is the $q$-Steiner system $\mathcal{S}_2(2,3,13)$, which was constructed by Braun et al. \cite{steiner}. Of particular significance for the present work, Itoh \cite{Itoh} developed in 1998 a recursive construction which has been used to obtain new infinite families of subspace $2$-designs. Schram \cite{Schram} presented a construction of subspace $2$-designs via $q$-GDDs. By combining the constructions of Itoh and Schram, Wang and Zhou \cite{WZ} obtained plenty of new infinite families of simple subspace $2$-designs. In fact, most of the aforementioned works focus primarily on block dimension \(k=3\). When the block dimension increases, the known existence and constructions of simple subspace \(2\)-designs become much less, and even the constructions of non-simple subspace \(2\)-designs is already a difficult problem.
	
	\begin{thm}\rm{\cite{Itoh,WZ}}
		Let $q$ be a prime power and let $l,m$ be positive integers. Suppose there exists a nontrivial $2$-$(l,3,\lambda)_q$ design admitting a Singer cycle $H\leq \GL(m,q^l)$ as an automorphism group, where $\lambda= sq(q+1)(q^3-1)+tq(q^2-1)$ for some integers $s,t$ with $t\in\left\lbrace 0,1\right\rbrace $ if $3\mid l$, and $t=0$ if $3\nmid l$. Then there exists a nontrivial $2$-$(ml,3,\lambda)_q$ design admitting $\GL(m,q^l)$ as an automorphism group.
	\end{thm}

Now we display the fundamental relationship between $q$-GDDs and subspace $2$-designs.

	\begin{thm}\rm{\cite{buratti}}
		If there exists a $2$-$(v+1,k,1)_q$ design, then there exists a $(v,k-1,k,q^2)_q$-GDD.
	\end{thm}
	
	\begin{thm}\rm{\cite{Schram}}\label{aaabb}
		Suppose that $(V,\mathcal{G},\mathcal{B})$ is a $(v,m,k,\lambda)_q$-GDD and $(Z,\mathcal{D})$ is a $2$-$(m+n,k,q^{\,n(k-2)}\lambda)_q$ design. Assume further that $Z$ contains an $n$-subspace $U$ such that $(U,\mathcal{D}_U)$ forms a $2$-$(n,k,q^{\,n(k-2)}\lambda)_q$ design, where $\mathcal{D}_U = \{\, B \in \mathcal{D} : B \subseteq U \,\}.$ Then there exists a $2$-$(v+n,k,q^{\,n(k-2)}\lambda)_q$ design.
	\end{thm}

	 A \emph{\(q\)-analog of a \(t\)-wise balanced design}, also called a \emph{\(t\)-\((v,K,\lambda)_q\) design} or a \emph{\(t\)-wise balanced design over the finite field}, is a collection \(\mathcal{B}\) of subspaces of \(V\) whose dimensions belong to \(K\), such that every \(t\)-subspace of \(V\) is contained in exactly \(\lambda\) blocks of \(\mathcal{B}\). Obviously, if $K=\left\lbrace k\right\rbrace $, then it is a subspace $t$-design.
	In particular, for \(t=2\), a \(2\)-\((v,K,\lambda)_q\) design is referred to as a \emph{\(q\)-analog of a pairwise balanced design}, and briefly denoted by a $q$-PBD$(v, K, \lambda)$. Braun \cite{PBD} presented a construction of an infinite family of nontrivial \(t\)-\((v,K,\lambda)_q\) designs admitting the standard Borel subgroup as an automorphism group.
	\begin{thm}\rm{\cite{PBD}}
		Let \(t\) be a positive integer. There exists a $t$-$(t+4,\{t+1,t+2\},q^3+q^2+q+1)_q$ design.
	\end{thm}
	
	Motivated by the preliminary findings in \cite{Itoh,Wz,WZ}, we pursue a deeper investigation into the orbit incidence matrix under the action of  a general linear group in this paper. Through a more rigorous and comprehensive characterization of this matrix, we are able to generalize the results of \cite{Wz,WZ} and establish a broad range of new infinite families of $q$-GDDs.

	The rest of the paper is organized as follows. In Section \ref{section2}, we collect several preliminary results on general linear groups and Singer cycles; we recall the Kramer--Mesner method and introduce the notation used throughout the paper. In Section \ref{section3}, we give a complete description of the action of \(\GL(m,q^l)\) on \(\Omega_k^{k-1}\), where $3\leq k\leq \min\left\lbrace m+1,l\right\rbrace $ and \(\Omega_k^{k-1}\) is the set of \(k\)-subspaces of $\GF(q)^{ml}$ whose \(\GF(q^l)\)-span has dimension \(k-1\). In particular, we compute the number of $G$-orbits on $\Omega_k^{k-1}$ and the length of every $G$-orbit. In Section \ref{section4}, we explicitly determine the submatrices of $\GL(m,q^l)$-incidence matrix between $2$-subspaces and $k$-subspaces. In Section \ref{section5}, we apply a $\GL(m,q^l)$-incidence submatrix to construct new infinite families of simple \(q\)-GDDs with arbitrary block dimension. We further present a recursive construction for simple $q$-analogs of pairwise balanced designs and consequently construct several new infinite families of such designs. Combining the above two constructions, we also obtain numerous infinite families of non-simple subspace $2$-designs. In Section \ref{section6}, we summarize the main results of the paper and conclude with some final remarks.

	\section{Preliminaries}\label{section2}
	
	 For $1 \le k \le v$, we denote by $\left[ V\atop k\right]_q $ the set of all $k$-subspaces of $V$ and this set is also called the \emph{Grassmannian}. It is well-known that the cardinality of $\left[ V\atop k\right]_q $ is denoted by the \emph{Gaussian binomial coefficient}
	\[
	\left[ v\atop k\right]_q
	= \frac{(q^{v}-1)(q^{v-1}-1)\cdots(q^{v-k+1}-1)}
	{(q^{k}-1)(q^{k-1}-1)\cdots(q-1)}.
	\]
	The subscript $q$ can be omitted if no confusion arises.
	For a simple $t$-$(v,K,\lambda)_q$ design $(V,\mathcal{B})$, its \emph{supplementary design} is $\left( V, \bigcup_{k\in K}\left[ V\atop k\right] \setminus \mathcal{B}\right),$ which is a $t$-$(v,K,\lambda')_q$ design with $\lambda'=\sum_{k\in K}\left[ v-t\atop k-t\right] -\lambda.$

	\subsection{General linear groups and Singer cycles}\label{GS}
	The \emph{general linear group $\GL(V)$}, also denoted by $\GL(v,q)$, consists of all invertible $\GF(q)$-linear transformations of $V$. After fixing a basis of $V$, we identify $\GL(V)$ with the group of all invertible $v\times v$ matrices over $\GF(q)$. Its cardinality is
	\[
	|\GL(v,q)|
	= \prod_{i=0}^{v-1} \left(q^{v} - q^{i}\right) .
	\]
	 Let $W \leq V$, $G\leq \GL(v,q)$. Define the \emph{$G$-orbit of $W$} by
	$$W^G=\left\lbrace g(W):g\in G\right\rbrace $$
	and the \emph{stabilizer of $W$ in $G$} by
	$$G_W=\left\lbrace g\in G:g(W)=W\right\rbrace.$$
	
	Fixing a $\GF(q)$-linear isomorphism $V \cong \GF(q^v)$, we may identify $V$ with the
	field $\GF(q^v)$ considered as a $v$-dimensional vector space over $\GF(q)$.  Let
	$w \in \GF(q^v)^{*}:=\GF(q^v)\setminus \left\lbrace 0\right\rbrace $ be a primitive element.  Then the map
	\[
	\begin{aligned}
		T_{w}: \GF(q^v) &\to \GF(q^v)\\
		 x &\mapsto wx
	\end{aligned}
	\]
	is $\GF(q)$-linear, and its matrix with respect to a fixed $\GF(q)$-basis of
	$\GF(q^v)$ is an element of $\GL(v,q)$.  The cyclic subgroup
	\[
	\langle T_{w} \rangle \;\le\; \GL(v,q),
	\]
	called a \emph{Singer cycle}, has order $q^{v}-1$ and acts regularly on $V \setminus \{0\}$. The induced action of a Singer cycle on the set of $1$-subspaces of $V$ is sharply transitive.
	Singer cycles are unique up to conjugacy in $\GL(v,q)$. We record the following facts on the action of a Singer cycle on subspaces of $V$.
\begin{thm}\rm{\cite{singer}}\label{sin}
	Let \(V\) be a \(v\)-dimensional vector space over \(\GF(q)\) and let \(H\) be a Singer cycle of $\GL(v,q)$.
	\begin{enumerate}
		\item [\rm(i)] 	For any $d$-subspace $W\leq V$, the stabilizer $H_W$ is isomorphic to \(\GF(q^u)^*\) for some $u\mid \gcd(d,v)$.
		\item [\rm(ii)] The number of $H$-orbits on the set of \(d\)-subspaces of \(V\) is
		\begin{equation}\label{ndv}
			n_d^v:=\frac{1}{q^v-1}
			\sum_{t\mid \gcd(d,v)}
			\left[ v/t\atop d/t\right]_{q^t}
			\left(
			\sum_{u\mid t}\mu\!\left(t/u\right)(q^u-1)
			\right),
		\end{equation}
		where \(\mu\) denotes the Möbius function.
		\item [\rm(iii)] For each $u\mid\gcd(d,v)$, the number of \(H\)-orbits on the \(d\)-subspaces of \(V\) whose stabilizer is isomorphic to \(\GF(q^u)^*\) is exactly
		\begin{equation}\label{ndu}
			n_{d,u}^v :=
			\frac{q^u-1}{q^v-1}
			\sum_{t:\,u\mid t\mid \gcd(d,v)}
			\mu\!\left(t/u\right)
			\left[ v/t\atop d/t\right]_{q^t}.
		\end{equation}
	\end{enumerate}
\end{thm}

\subsection{Kramer–Mesner method }

An element \(g \in \GL(v,q)\) is called an \emph{automorphism of a \(t\)-\((v,k,\lambda)_q\) design} \(\left( V,\mathcal{B}\right) \) if and only if
\[
\mathcal{B}=g\left( \mathcal{B}\right) :=\{g(S)\mid S\in\mathcal{B}\}.
\]
The set of all automorphisms of \(\left(V, \mathcal{B}\right) \) forms a group, called the \emph{full automorphism group of \(\left( V, \mathcal{B}\right) \)}. Any subgroup of the full automorphism group is called an \emph{automorphism group}.

Note that a \(t\)-\((v,k,\lambda)_q\) design \(\left(V,\mathcal{B}\right) \) admits a subgroup \(G\le \GL(v,q)\) as an automorphism group if and only if \(\mathcal{B}\) is a union of some \(G\)-orbits on
$\left[ V\atop k\right] .$
This leads to a construction method analogous to the Kramer--Mesner construction for classic \(t\)-designs \cite{KM}.

Let $G\leq \GL(v,q)$ act on $V=\GF(q)^v$ and this induces an action of $G$ on subspaces of $V$. Take a complete set of representatives $\mathcal R_i$ of the $G$-orbits on $\left[ V\atop i \right]$, $1\leq i\leq v$. Given $1\leq t\leq k< v$, the \emph{$G$-incidence matrix $M_{t,k}$} between $\left[ V\atop t \right] $ and $\left[ V\atop k \right] $, is defined as an $\left| \mathcal R_t \right|\times \left| \mathcal R_k \right| $ matrix, whose rows and columns are indexed by $\mathcal R_t$ and $\mathcal R_k$, respectively, and the $(T,K)$-entry is
 \[\bigl|\{\,K' \in K^G\mid T\subseteq K'\,\}\bigr|.\]
The following is immediate.

\begin{thm}\label{KM1}
	A \(t\)-\((v,k,\lambda)_q\) design admitting \(G\le \GL(v,q)\) as an automorphism group exists if and only if there is a nonnegative integer column vector \(\boldsymbol x\) such that
	\[
	M_{t,k}\boldsymbol x=\lambda \boldsymbol{1},
	\]
	where \(\boldsymbol{1}\) denotes the all-one column vector. Furthermore, if $\boldsymbol x$ is a $0$-$1$ column vector, then a simple $t$-$(v,k,\lambda)_q$ design admitting \(G\) as an automorphism group exists.
\end{thm}

	\subsection{Notation}
	This paper constructs $q$-GDDs over the vector space $V=\GF(q)^{ml}$, admitting $G=\GL(m,q^l)$ as an automorphism group. The incidence matrix $M_{2,k}$ will be studied in detail under the action of $\GL(m,q^l)$, which is closely related with the action of a Singer cycle on subspaces of $\GF(q)^l$.
	
	We view $V=\GF(q)^{ml}$ as an $m$-dimensional vector space over $\GF(q^l)$.
	Let $\left\lbrace  Y_{1},Y_{2},\dots,Y_{m}\right\rbrace $ be a basis
	of $V$ over $\GF(q^l)$.
	With respect to this basis, we identity each element $g\in G$ with its matrix over $\GF(q^l)$, that is,
	\[
	\left( g(Y_1),\dots,g(Y_m)\right)
	=
	\left( Y_1,\dots,Y_m\right) g.
	\]
	We also regard $\GF(q^l)$ as an $l$-dimensional vector space over $\GF(q)$ and fix a basis
	$\left\lbrace x_1(=1),x_2,\dots,x_l\right\rbrace $ of $\GF(q^l)$ over $\GF(q)$. It follows that
	\begin{equation}\label{V}
	V\cong\left\langle  x_iY_j: 1\le i\le l,\ 1\le j\le m\right\rangle_{\GF(q)} .
	\end{equation}
	For convenience, we also identify $V$ with the field $\GF(q^{ml})$ via a fixed $\GF(q)$-linear isomorphism.
	
	 Throughout the rest of the paper, let $m,l,k$ be integers with $3\leq k\leq\min\left\lbrace m+1,l \right\rbrace $ and the following notation will be adopted, unless otherwise stated.
	
	\begin{itemize}
	    \item $V=\GF(q)^{ml}\cong\left\langle  x_iY_j : 1\le i\le l,\ 1\le j\le m\right\rangle_{\GF(q)}$ by \eqref{V};
		\item $G=\GL(m,q^l)$ and $H$ is a Singer cycle of $\GF(q^l)$;
		\item $\langle W \rangle_{\GF(q)} \left(\text{resp.} \  \langle W \rangle_{\GF(q^l)} \right) $  denotes the $\GF(q)$-subspace $\left( \text{resp.}\ \GF(q^l)\text{-subspace}\right)$  of $V$ spanned by the subset $W \subseteq V$. For brevity, $\left\langle W\right\rangle_{\GF(q)} $ is also simply denoted $\left\langle W \right\rangle $;
		\item $\Omega_i^j = \left\{ W \leq V \,:\, \dim_{\GF(q)}(W) = i,\ \dim_{\GF(q^l)}\left(\langle W \rangle_{\GF(q^l)}\right) = j \right\}$;
		\item $\mathcal{F}_i^j$ denotes a complete set of representatives for the $G$-orbits on $\Omega_i^j$;
		\item $\mathcal{O}_i$ denotes a complete set of representatives for the $H$-orbits on $\left[  \GF(q)^l \atop i \right] $;
		\item $M_{t,k}$ denotes the $G$-incidence matrix between $\left[ V\atop t \right] $ and $\left[ V\atop k \right] $, where the rows and columns are indexed by complete sets of representatives of the corresponding $G$-orbits;
		\item $\boldsymbol{0}$ denotes a zero matrix or a zero vector;
		\item $M_{r \times s}(q)$ denotes the set of all $r \times s$ matrices over $\GF(q)$;
		\item $n_d^v$ denotes the number of $H$-orbits on $\left[ \GF(q)^v\atop d \right] $, and $n_{d,u}^v$ denotes the number of $H$-orbits on $\left[ \GF(q)^v\atop d \right] $ with orbit stabilizer isomorphic to $\GF(q^u)^*$, see \eqref{ndv} and \eqref{ndu}.
	\end{itemize}

	\section{$G$-orbits on $\Omega_{k}^{k-1}$}\label{section3}

	Recall that $\Omega_i^j$ consists of all $i$-subspaces of $V$ whose span over $\GF(q^l)$ has dimension $j$. In this section,
 	 we partition the \(G\)-orbits on \(\Omega_{k}^{\,k-1}\) into $k-1$ classes, namely, $\Omega_{k,1}^{k-1}\cup\cdots\cup\Omega_{k,k-1}^{k-1}$,
	represented by subspaces of the form \(T(u_{1},\dots,u_{r})\), where \(1\le r\leq k-1\).
	By analysis on the relationship between $G$-orbits and $H$-orbits, we determine the number of $G$-orbits on $\Omega_k^{k-1}$ and the size of stabilizer of every $k$-subspace in $\Omega_{k}^{k-1}$.
	
\subsection{Orbit decomposition of $\Omega_{k}^{k-1}$}
	For $1\leq r\leq k-1$, let \(u_{1},\dots,u_{r}\in \GF(q^{l})\) such that \(1,u_{1},\dots,u_{r}\) are linearly independent
	over \(\GF(q)\). Then the vectors $Y_{1},\dots,Y_{k-1}, u_{1}Y_{1}+\cdots+u_{r}Y_{r}\in V$ are linearly independent over \(\GF(q)\). Define
	\begin{equation}\label{Tur}
		\begin{aligned}
			T(u_{1},\dots,u_{r})
			&:= \bigl\langle Y_{1},\dots,Y_{k-1},\,u_{1}Y_{1}+\cdots+u_{r}Y_{r}\bigr\rangle\\
			&= \left\{\;\sum_{i=1}^{r}(a_{i}+bu_{i})Y_{i}+\sum_{i=r+1}^{k-1}a_{i}Y_{i}
			\;:\; a_{1},\dots,a_{k-1},b\in \GF(q)\right\}.
		\end{aligned}
	\end{equation}
	Note that for \(1\le i\le r\), the coefficient of \(Y_{i}\) lies in
	\(\langle 1,u_{i}\rangle\subseteq \GF(q^{l})\). Hence \(T(u_{1},\dots,u_{r})\in \Omega_{k}^{\,k-1}\).
	Finally, we denote by $\Omega_{k,r}^{\,k-1}$ the union of all $G$-orbits $T(u_1, \dots, u_r)^G$, where $1, u_1, \dots, u_r$ are linearly independent over $\GF(q)$. Correspondingly, denote by $\mathcal F_{k,r}^{\,k-1}$ a complete set of representatives of $G$-orbits on $\Omega_{k,r}^{\,k-1}$.

\begin{lem}\label{Omega_k_k-1_orbits}
	For \(1\le r\le k-1\), we have the disjoint decomposition
	\[
	\Omega_k^{k-1}=\bigsqcup_{r=1}^{k-1}\Omega_{k,r}^{k-1}
	=
	\bigsqcup_{r=1}^{k-1}\ \bigsqcup_{T\in\mathcal{F}_{k,r}^{k-1}} T^G.
	\]
\end{lem}

\begin{proof}
	
	We only prove that every element of \(\Omega_k^{k-1}\) lies in one of these orbits. Let \(W\in\Omega_k^{k-1}\). Then $\dim_{\GF(q)}\left( W \right)=k $, \(\dim_{\GF(q^l)}\left( \left\langle W\right\rangle_{\GF(q^l)} \right) =k-1\). Choose \(z_1,\dots,z_{k-1}\in W\) such that $\left\langle z_1,\dots,z_{k-1}\right\rangle_{\GF(q^l)}=\left\langle W\right\rangle_{\GF(q^l)} $. Since \(\dim_{\GF(q)}W=k\), there exists $z\in W\setminus \langle z_1,\dots,z_{k-1}\rangle$, such that $\left\langle z_1,\dots,z_{k-1},z\right\rangle=W $. Write
	\[
	z=\alpha_1z_1+\cdots+\alpha_{k-1}z_{k-1},
	\]
	where $\alpha_i\in\GF(q^l)$. Since \(z\notin \langle z_1,\dots,z_{k-1}\rangle\), not all \(\alpha_i\) belong to \(\GF(q)\). Without loss of generality, assume that $\dim_{\GF(q)}\left( \left\langle 1,\alpha_1,\dots,\alpha_{k-1}\right\rangle \right)=r+1 $ and $\left\langle 1,\alpha_1,\dots,\alpha_{k-1}\right\rangle=\left\langle 1,\alpha_1,\dots,\alpha_{r}\right\rangle$, for some $1\leq r\leq k-1$.
	
	Since \(G=\GL(m,q^l)\), there exists \(g\in G\) such that
	\[
	g(\left\langle W\right\rangle_{\GF(q^l)})=\langle Y_1,\dots,Y_{k-1}\rangle_{\GF(q^l)}\ \text{and}\ g(z_i)=Y_i\ \text{for}\ 1\leq i\leq k-1.
	\]
	Moreover, $g(W)=T(\alpha_1,\dots,\alpha_r)\in \Omega_{k,r}^{k-1}$, i.e., $W\in T^G$ for some $T\in \mathcal F_{k,r}^{\,k-1}$.
	We can easily prove that $\Omega_{k,r}^{k-1}\cap \Omega_{k,r'}^{k-1}=\emptyset$ for any $r\neq r'$.  Hence the conclusion follows.
\end{proof}

\subsection{Relationship between $\mathcal F_{k,r}^{k-1}$ and $\mathcal O_{r+1}$}
In this subsection, we show that there is a one-to-one correspondence between $G$-orbits on $\Omega_{k,r}^{k-1}$ and $H$-orbits on $\left[ \GF(q)^l\atop r+1\right] $.

	\begin{lem}\label{mapOur}
	For $1\leq r\leq k-1$, let $\left\langle 1,u_1,\dots,u_r \right\rangle$ and $\left\langle 1,u_1^{\prime},\dots, u_r^{\prime} \right\rangle $ be two $(r+1)$-subspaces of $\mathrm{GF}(q^l)$ over $\GF(q)$. Then the $G$-orbit $T(u_1,\dots,u_r)^G$ contains $T(u_1^{\prime},\dots, u_r^{\prime})$ if and only if the $H$-orbit $\left\langle 1,u_1,\dots, u_r \right\rangle^H$ contains $\left\langle 1,u_1^{\prime},\dots, u_r^{\prime} \right\rangle$.
	\end{lem}
	\begin{proof}
		First we prove the necessity. Let $g$ be an element of $G$ such that
		\begin{equation*}\label{eqgOr}
			g(T(u_1,\dots,u_r)) = T(u_1^{\prime},\dots,u_r^{\prime}).
		\end{equation*}
		From the definition of the $T(u_1,\dots,u_r)$ in \eqref{Tur}, $g$ has the form
			$$
		g = \left( \begin{array}{cc}\begin{array}{cccc}
				a_{1,1}+b_1u_1'   & \cdots & a_{1,k-1}+b_{k-1}u_1'\\
				\vdots & \vdots & \vdots \\
				a_{r,1}+b_1u_r'   & \cdots & a_{r,k-1}+b_{k-1}u_r'\\		
				a_{r+1,1}  & \cdots & a_{r+1,k-1} \\
				\vdots  & \vdots & \vdots \\
				a_{k-1,1}  & \cdots &  a_{k-1,k-1}
			\end{array} &B\\
			\boldsymbol{0} & C
			
		\end{array}  \right),
		$$
		where $a_{ij}, b_i \in \mathrm{GF}(q)$ for $1 \le i,j \le k-1$, $B \in M_{(k-1) \times (m-k+1)}(q^l)$, $C \in \mathrm{GL}(m-k+1,q^l)$. Then $$ \begin{aligned}
			&g(u_1 Y_1+\dots+ u_r Y_r)\\=& \sum_{i=1}^{r}\left(\sum_{j=1}^{r}(a_{ij}+b_ju_i')u_j \right)Y_i+\sum_{i=r+1}^{k-1}\left(\sum_{j=1}^{r}a_{ij}u_j \right)Y_i  .
		\end{aligned} $$ Since $g(u_1Y_1+\dots+ u_rY_r) \in T(u_1',\dots, u_r')$, we have
		$$
			\begin{aligned}
				&	\sum_{j=1}^{r}(a_{ij}+b_ju_i')u_j \in \langle 1, u_i'\rangle \text{ for } 1 \le i \le r;\\
				& \sum_{j=1}^{r}a_{ij}u_j\in \GF(q) \text{ for } r+1 \le i \le k-1.
			\end{aligned}
		$$
		Note that $1,u_1,\dots,u_r$ are linearly independent over $\GF(q)$,  so $a_{ij}=0$ for $r+1 \le i \le k-1, 1 \le j\le r$.
		
		Then
		$$g = \begin{pmatrix}
			\begin{matrix}
				D&E \\
				\boldsymbol{0}&F
			\end{matrix} & B \\
			\boldsymbol{0} & C
		\end{pmatrix},$$ where
		$$D = (a_{ij}+b_ju_i')_{1\leq i, j\leq r},\  E=(a_{ij}+b_iu_j')_{\substack{1\leq i\leq r\\ r+1\leq j\leq k-1}},\  F=(a_{ij})_{\substack{r+1\leq i,j\leq k-1\\ }} .$$
		Since $g\in G$, we have
		$\operatorname{det}(D) \ne 0$ and
		\[\begin{aligned}
			D: \left\langle Y_1,\dots,Y_{r},\sum_{i=1}^{r}u_iY_i\right\rangle \to \left\langle Y_1,\dots,Y_{r},\sum_{i=1}^{r}u_i'Y_i\right\rangle
		\end{aligned}
		\]
		is a $\GF(q^l)$-linear mapping.
		We claim that the action of \(D\) is invertible over $\GF(q)$. Note that the action of $D$ on $\left\langle Y_1,\dots,Y_r \right\rangle $ is invertible. We only need to prove that $D\Bigl(\sum_{j=1}^{r}u_jY_j\Bigr)$ cannot be linearly expressed by $D(Y_1),\dots,D(Y_r)$ over $\GF(q)$.
		Otherwise, there exist $c_1,\dots,c_r\in\GF(q)$ such that
		\[
		D\Bigl(\sum_{j=1}^{r}u_jY_j\Bigr)=\sum_{j=1}^{r}c_jD(Y_j).
		\]
		 As $D$ acts on $\left\langle Y_1,\dots,Y_r\right\rangle $ as an invertible $\GF(q^l)$-linear transformation obviously, we have $D(\sum_{j=1}^{r}\left( u_j-c_j\right) Y_j)=0$, i.e., $u_j=c_j\in\GF(q)$ for $1\leq j\leq r$, contradicting that $1,u_1,\dots,u_r$ are linearly independent over $\GF(q)$. Thus the claim is true.
		
		 Note that
		 \begin{align}
		 	&D(Y_j)=\sum_{i=1}^{r}\left( a_{ij}+b_ju_i'\right) Y_i\ \text{for}\ 1\leq j\leq r ,\label{zy1}\\
		 	&D\left( \sum_{i=1}^{r}u_iY_i\right) =\sum_{i=1}^{r}d_iY_i+\alpha\sum_{i=1}^{r}u_i'Y_i\ \text{for some}\ \alpha,d_1,\dots,d_r\in\GF(q).\label{zy2}
		 \end{align}
		  Set $$A=(a_{ij})_{r\times r}, \boldsymbol d=(d_1,\dots,d_r)^T,\boldsymbol b=(b_1,\dots,b_r)^T.$$
		Then we have
		\[D\left(Y_1,Y_2,\dots,Y_r,\sum_{i=1}^{r}u_iY_i \right) =\left(Y_1,Y_2,\dots,Y_r,\sum_{i=1}^{r}u_i'Y_i \right)\begin{pmatrix}
		A&\boldsymbol d\\
		\boldsymbol b^T&\alpha
		\end{pmatrix}.
		\]
		Hence, $\begin{pmatrix}
			A&\boldsymbol d\\
			\boldsymbol b^T&\alpha
		\end{pmatrix}\in\GL(r+1,q)$ by the above claim. Denote
		\begin{equation}\label{uu}
		\boldsymbol u=(u_1,\dots,u_r)^T,\
		\boldsymbol u'=(u_1',\dots,u_r')^T.
		\end{equation}
	
		From \eqref{zy1}, we have
		\[D\left( \sum_{i=1}^{r}u_iY_i\right) =\sum_{i=1}^{r}\left(\sum_{j=1}^{r}(a_{ij}+b_ju_i')u_j \right)Y_i.
		\]
		Comparing the coefficients of $Y_1,\dots,Y_r$ with \eqref{zy2}, we obtain
		\[
		A\boldsymbol u+(\boldsymbol b^T\boldsymbol u)\boldsymbol u'
		=
		\boldsymbol d+\alpha \boldsymbol u',
		\]
		that is,
		\begin{equation}\label{eq:Au-d}
		A\boldsymbol u-\boldsymbol d
		=
		\lambda\boldsymbol u',\ \text{where}\ \lambda=\alpha-\boldsymbol b^T\boldsymbol u.
		\end{equation}
		Then, using \eqref{eq:Au-d}, we have the identity
		\[
		\begin{pmatrix}
		1 & \boldsymbol{0}\\
		\boldsymbol u' & I_r
		\end{pmatrix}
		\begin{pmatrix}
		\alpha & \boldsymbol b^T\\
		\boldsymbol d & A
		\end{pmatrix}
		\begin{pmatrix}
		1 & \boldsymbol{0}\\
		-\boldsymbol u & I_r
		\end{pmatrix}
		=
		\begin{pmatrix}
		\lambda & \boldsymbol b^T\\
		\boldsymbol{0} & A+\boldsymbol u'\boldsymbol b^T
		\end{pmatrix}.
		\]
		The left-hand side is invertible, hence so is the right-hand side. In particular, $\lambda\neq 0.$
		Moreover, \eqref{eq:Au-d} shows that
		\[
		\lambda,\lambda u_i' \in \langle 1,u_1,\dots,u_r\rangle\ \text{for}\ 1\leq i\leq r.
		\]
		It follows that
		\[
		\lambda\langle 1,u_1',\dots,u_r'\rangle
		=
		\langle 1,u_1,\dots,u_r\rangle.
		\]
		Therefore the $H$-orbit $\langle 1,u_1,\dots,u_r\rangle^H$ contains
		$\langle 1,u_1',\dots,u_r'\rangle$.

	    We next prove the sufficiency. Assume that the $H$-orbit
		$\langle 1,u_1,\dots,u_r\rangle^H$ contains
		$\langle 1,u_1',\dots,u_r'\rangle$. Then there exists $w\in \GF(q^l)^*$ such that
		\[
		w \langle 1,u_1,\dots,u_r\rangle
		=
		\langle 1,u_1',\dots,u_r'\rangle.
		\]
		Let $\lambda:=w^{-1}$. Then
		\[
		\lambda\langle 1,u_1',\dots,u_r'\rangle
		=
		\langle 1,u_1,\dots,u_r\rangle.
		\]
		Hence the multiplication by $\lambda$ induces a $\GF(q)$-linear isomorphism from
		$\langle 1,u_1',\dots,u_r'\rangle$ onto $\langle 1,u_1,\dots,u_r\rangle$.
		Then
		\[
		\lambda(1,u_1',\dots,u_r')= (1,u_1,\dots,u_r)M\ \text{for some}\
		M=
		\begin{pmatrix}
			\alpha & \boldsymbol e^T\\
			\boldsymbol c & A
		\end{pmatrix}
		\in \GL(r+1,q),
		\]
		where $\alpha\in \GF(q)$, $\boldsymbol c=(c_1,\dots,c_r)^T\in \GF(q)^r$,
		$\boldsymbol e=(e_1,\dots,e_r)^T\in \GF(q)^r$, and $A=(a_{ij})\in M_{r\times r}(q)$. In other words, using \eqref{uu} we have
		\begin{equation}\label{eq:lambdau}
			\lambda=\alpha+\boldsymbol u^T \boldsymbol c,\
			\lambda\boldsymbol u'^T=\boldsymbol e^T+\boldsymbol u^T A.
		\end{equation}
		
		Define
		\begin{equation}\label{DDD}
		D:=A-\boldsymbol c\boldsymbol u'^T\in M_r(q^l).
		\end{equation}
		Then, by \eqref{eq:lambdau} we have
		\[
		\begin{pmatrix}
			1 & \boldsymbol u^T\\
		\boldsymbol{0}	 & I_r
		\end{pmatrix}
		\begin{pmatrix}
			\alpha & \boldsymbol e^T\\
			\boldsymbol c & A
		\end{pmatrix}
		\begin{pmatrix}
			1 & -\boldsymbol u'^T\\
		\boldsymbol{0}	 & I_r
		\end{pmatrix}
		=
		\begin{pmatrix}
			\lambda & \boldsymbol{0}\\
			\boldsymbol c & D
		\end{pmatrix}.
		\]
		Since the left-hand side is invertible and $\lambda\neq 0$, we conclude that
		$D\in \GL(r,q^l)$.
		
		Now define
		\[
		g=\begin{pmatrix}
			D & \boldsymbol{0}\\
			\boldsymbol{0} & I_{m-r}
		\end{pmatrix}.
		\]
		Obviously, $g\in G$ and
		\[
		g(Y_j)=\sum_{i=1}^r\left( a_{ij}-c_iu_j'\right) Y_i \in T(u_1',\dots,u_r')\ \text{for}\ 1\le j\le r.
		\]
		Moreover, from \eqref{eq:lambdau} and \eqref{DDD},
		\[
		\boldsymbol e^T+\boldsymbol u^T D
		=
		\boldsymbol e^T+\boldsymbol u^T A-\boldsymbol u^T\boldsymbol c\boldsymbol u'^T
		=
		\alpha\boldsymbol u'^T,
		\]
		and therefore
		\[
		\begin{aligned}
			g\left( \sum_{i=1}^r u_iY_i\right)
			&=\sum_{i=1}^{r}u_ig(Y_i)\\
			&=\sum_{i=1}^r\left(\sum_{j=1}^r (a_{ij}-c_iu_j')u_j \right)Y_i \\
			&= \sum_{i=1}^r (-e_i+\alpha u_i')Y_i
			\in T(u_1',\dots,u_r').
		\end{aligned}
		\]
		Also note that
		\[
		g(Y_i)=Y_i\in T(u_1',\dots,u_r')\ \text{for}\ r+1\le i\le k-1.
		\]
		Hence
		\[
		g\bigl(T(u_1,\dots,u_r)\bigr)= T(u_1',\dots,u_r').
		\]
		Therefore, the $G$-orbit $T(u_1,\dots,u_r)^G$ contains $T(u_1',\dots,u_r')$.
	\end{proof}
	
    \begin{cor}\label{num}
    	For $1\leq r\leq k-1$, the number of $G$-orbits on $\Omega_{k,r}^{k-1}$ equals
   \[
   \lvert \mathcal{F}_{k,r}^{k-1}\rvert = n_{r+1}^l.
   \]
   In particular, given a positive integer $u\mid \gcd(d,l)$, the number of \(G\)-orbits containing some $T(u_1,\dots,u_r)\in \Omega_{k,r}^{k-1}$ with $H_{\left\langle 1,u_1\dots,u_r\right\rangle }\cong\GF(q^u)^*$ equals \(n_{r+1,u}^l\).
   \end{cor}
    \begin{proof}
    This conclusion follows directly from Theorem \ref{sin} and Lemma \ref{mapOur}.
    \end{proof}

\subsection{The stabilizer of $T(u_1,\dots,u_r)$ in $G$}

\begin{lem}\label{vvv}
	Let $1,u_1,\dots,u_r$ be elements of $\GF(q^l)$ which are linearly independent over $\GF(q)$. Define an $r\times s$ matrix
	\[
	g=(a_{ij}+b_j u_i)_{r\times s},
	\]
	where $a_{ij},b_j\in \GF(q)$ and $1\le s\le r$. For each $1\le j\le s$, define
	\[
	\boldsymbol v_j:=(b_j,a_{1j},a_{2j},\dots,a_{rj})\in \GF(q)^{r+1}.
	\]
	Then the $s$ columns of $g$ are linearly independent over $\GF(q^l)$ if and only if the vectors $\boldsymbol v_1,\dots,\boldsymbol v_s$ are linearly independent over $\GF(q)$.
\end{lem}

\begin{proof}
	
	We first prove the necessity. Suppose on the contrary that $\boldsymbol v_1,\dots,\boldsymbol v_s$ are linearly dependent over $\GF(q)$. Then there exist $\lambda_1,\dots,\lambda_s\in \GF(q)$, not all zero, such that
	\[
	\sum_{j=1}^s \lambda_j \boldsymbol v_j=\boldsymbol 0.
	\]
	Comparing coordinates yields
	\[
	\sum_{j=1}^s \lambda_j b_j=0
	\ \text{and}\
	\sum_{j=1}^s \lambda_j a_{ij}=0
	\ (1\le i\le r).
	\]
	Hence, for each $1\le i\le r$,
	\[
	\sum_{j=1}^s \lambda_j(a_{ij}+b_j u_i)
	=
	\sum_{j=1}^s \lambda_j a_{ij}
	+
	u_i\sum_{j=1}^s \lambda_j b_j
	=0.
	\]
	Therefore the columns of $g$ are linearly dependent over $\GF(q^l)$, a contradiction.
	
	We next prove the sufficiency. Assume that $\boldsymbol v_1,\dots,\boldsymbol v_s$ are linearly independent over $\GF(q)$.
	Suppose on the contrary that the columns of $g$ are linearly dependent over $\GF(q^l)$ and hence every $s\times s$ minor of $g$ is zero. Let
	\[
	I=\{i_1,\dots,i_s\}\subseteq \{1,\dots,r\},\ i_1<\cdots<i_s.
	\]
	Consider the $s\times s$ submatrix of $g$ formed by the rows indexed by $I$. Its determinant is
    \[\begin{aligned}
    	 0&=\begin{vmatrix}
    		a_{i_1,1}+b_{1}u_{i_1}&\cdots &a_{i_1,s}+b_su_{i_1}\\
    		\vdots &\vdots &\vdots\\
    		a_{i_{s},1}+b_1u_{i_{s}}&\cdots&a_{i_{s},s}+b_su_{i_{s}}
    	\end{vmatrix}\\
    	&=\begin{vmatrix}
    		a_{i_1,1}&\cdots &a_{i_1,s}\\
    		\vdots &\vdots &\vdots\\
    		a_{i_{s},1}&\cdots&a_{i_{s},s}
    	\end{vmatrix}+
    	\begin{vmatrix}
    		b_1 &\cdots & b_s\\
    		a_{i_2,1}&\cdots &a_{i_2,s}\\
    		\vdots &\vdots &\vdots\\
    		a_{i_{s},1}&\cdots&a_{i_{s},s}
    	\end{vmatrix}u_{i_1}+\cdots+
    \begin{vmatrix}
    	a_{i_1,1}&\cdots &a_{i_1,s}\\
    	\vdots &\vdots &\vdots\\
    	a_{i_{s-1},1}&\cdots &a_{i_{s-1},s}\\
    	b_1&\cdots&b_s
    	\end{vmatrix}u_{i_s}.
    \end{aligned}
    \]
Indeed, since $1,u_{i_1}\dots,u_{i_s}$ are linearly independent over $\GF(q)$, it follows that
\[\begin{aligned}
	0&=\begin{vmatrix}
		a_{i_1,1}&\cdots &a_{i_1,s}\\
		\vdots &\vdots &\vdots\\
		a_{i_{s},1}&\cdots&a_{i_{s},s}
	\end{vmatrix}=\begin{vmatrix}
		b_1 &\cdots & b_s\\
		a_{i_2,1}&\cdots &a_{i_2,s}\\
		\vdots &\vdots &\vdots\\
		a_{i_{s},1}&\cdots&a_{i_{s},s}
	\end{vmatrix}
	=\cdots=\begin{vmatrix}
	a_{i_1,1}&\cdots &a_{i_1,s}\\
	\vdots &\vdots &\vdots\\
	a_{i_{s-1},1}&\cdots &a_{i_{s-1},s}\\
	b_1&\cdots&b_s
	\end{vmatrix}.
\end{aligned}
\]
The equalities above show that every $s\times s$ minor of matrix $(\boldsymbol v_1^T,\dots,\boldsymbol v_s^T)$ is zero. Hence $\boldsymbol v_1,\dots,\boldsymbol v_s$ are linearly dependent over $\GF(q)$, a contradiction.
\end{proof}

\begin{cor}\label{cla}
	Let $1,u_1,\dots,u_r$ be elements of $\GF(q^l)$ which are linearly independent over $\GF(q)$. Then the matrix $g=(a_{ij}+b_ju_i)_{r\times r}$ is invertible, where $a_{ij},b_j\in\GF(q)$ if and only if the vectors $\boldsymbol v_1,\dots,\boldsymbol v_r$ are linearly independent over $\GF(q)$, where $\boldsymbol v_j=(b_j,a_{1j},\dots,a_{rj})$, $1\leq i,j\leq r$.
\end{cor}

	\begin{lem}\label{G_Our}
		For $1\leq r\leq k-1$, let $\left\langle 1,u_1,\dots,u_r \right\rangle$ be an $(r+1)$-subspace of $\mathrm{GF}(q^l)$. Then the size of the stabilizer of $T(u_1,\dots,u_r)$ in $G$ is
		$$\Big|G_{T(u_1,\dots, u_r)} \Big| =
		\left| H_{\left\langle 1,u_1,\dots,u_r\right\rangle }\right|  \prod_{i=r+1}^{k-1}(q^k-q^i) \prod_{i=k-1}^{m-1}(q^{ml}-q^{il}).
		$$
	\end{lem}
	\begin{proof}
		By the similar argument to that in the proof of Lemma~\ref{mapOur}, an element $g$ of $G_{T(u_1,\dots, u_r)}$ can be written as
		$$
		g = \begin{pmatrix}
			A & B\\
			\boldsymbol{0} & C
			\end{pmatrix},
		$$
		where $$A =\left(  \begin{matrix}
			a_{1,1}+b_1u_1 & \cdots&a_{1,r}+b_ru_1&a_{1,r+1}+b_{r+1}u_1 &\cdots& a_{1,k-1}+b_{k-1}u_1 &\\
			\vdots & \vdots&\vdots&\vdots & \vdots &\vdots  \\
			a_{r,1}+b_1u_r & \cdots&a_{r,r}+b_ru_r&a_{r,r+1}+b_{r+1}u_r &\cdots& a_{r,k-1}+b_{k-1}u_r \\
			0  &\cdots&0&a_{r+1,r+1} &\cdots&a_{r+1,k-1}\\
			\vdots & \vdots&\vdots &\vdots & \vdots &\vdots\\
			0 & \dots&0&a_{k-1,r+1} & \cdots&a_{k-1,k-1}
		\end{matrix}\right)  $$
		for $a_{ij}, b_i \in \mathrm{GF}(q) \left( 1 \le i,j \le k-1\right) $, $B \in M_{(k-1) \times (m-k+1)}(q^l)$, $C \in \mathrm{GL}(m-k+1,q^l)$.
	 Set
		\begin{equation*}\label{eq:Gprime}
			G'=\left\{
			D=\begin{pmatrix}
				a_{11}+b_1u_1   & \cdots & a_{1r}+b_ru_1\\
				\vdots        & \ddots & \vdots\\
				a_{r1}+b_1u_r & \cdots & a_{rr}+b_ru_r
			\end{pmatrix}
			\,:\,
			\begin{array}{l}
				a_{ij},\,b_j\in\GF(q), \det(D)\neq 0,\\
				D\bigl(\sum_{i=1}^ru_iY_i\bigr)\in T'(u_1,\dots,u_r)
			\end{array}
			\right\},
		\end{equation*} where
		\begin{equation}\label{T'}
			T'(u_1,\dots,u_r)=\left\langle Y_1,\dots,Y_r,\sum_{i=1}^r u_iY_i\right\rangle.
		\end{equation}
			Obviously,
		\begin{equation}\label{Gs}
	\begin{aligned}
			\Big|G_{T(u_1,\dots, u_r)} \Big| &=\left| G'\right| q^{(r+1)(k-r-1)}\left|\GL(k-r-1,q) \right| q^{(k-1)(m-k+1)} \left| \GL(m-k+1,q^l) \right| \\
			&=\left| G'\right| \prod_{i=r+1}^{k-1}(q^k-q^i) \prod_{i=k-1}^{m-1}(q^{ml}-q^{il}).
		\end{aligned}
		\end{equation}
	Then the conclusion follows if Claim \ref{bij} is true.
	\end{proof}

	\begin{cla}\label{bij}
		There is a bijection from $G'$ to $H_{\left\langle 1,u_1,\dots,u_r\right\rangle }$.
	\end{cla}
	Now, we prove this claim.
	For any \(D\in G'\), we have
	\begin{equation}\label{gU1}
			D(\sum_{i=1}^r u_iY_i)= \sum_{i=1}^{r}\left(\sum_{j=1}^{r}(a_{ij}+b_ju_i)u_j \right)Y_i\in T'(u_1,\dots,u_r).
	\end{equation}
	By \eqref{T'}, there exists unique \(\alpha_1,\dots,\alpha_r,\beta\in\GF(q)\) such that
	\begin{equation}\label{gU2}
	D(\sum_{i=1}^r u_iY_i)
	=\sum_{i=1}^{r}(\alpha_i+\beta u_i)Y_i.
	\end{equation}
	Set
	\begin{equation*}
	\lambda:=\sum_{j=1}^{r}b_ju_j-\beta
	\in \langle 1,u_1,\dots,u_r\rangle.
	\end{equation*}
	Then, for each \(i=1,\dots,r\), comparing the coefficients of $Y_1,\dots,Y_r$ in \eqref{gU1} and \eqref{gU2} gives
	\begin{equation}\label{xxx}
	\lambda u_i
	=\alpha_i-\sum_{j=1}^{r}a_{ij}u_j
	\in \langle 1,u_1,\dots,u_r\rangle.
	\end{equation}
	Therefore \(\lambda\langle 1,u_1,\dots,u_r\rangle\subseteq \langle 1,u_1,\dots,u_r\rangle\).
	If \(\lambda=0\), then $\sum_{j=1}^{r}b_ju_j=\beta\in\GF(q).$
Because \(1,u_1,\dots,u_r\) are linearly independent over \(\GF(q)\), we obtain
	\(\beta=b_1=\cdots=b_r=0\). It follows from \eqref{gU1} and \eqref{gU2} that
	\[
	\sum_{j=1}^{r}a_{ij}u_j=\alpha_i\in\GF(q)
	\ (1\le i\le r),
	\]
	which again implies \(a_{ij}=0\) for all \(i,j\), contradicting \(\det(D)\neq 0\).
	Thus \(\lambda\neq 0\) and so $\lambda\in H_{\langle 1,u_1,\dots,u_r\rangle}.$ Hence, we have a well-defined mapping
	\[
	\begin{aligned}
		\psi: G' &\longrightarrow H_{\langle 1,u_1,\dots,u_r\rangle}\\
		D=(a_{ij}+b_ju_i)_{r\times r} &\longmapsto \lambda=\sum_{j=1}^{r}b_ju_j-\beta,\ \text{where}\ \eqref{gU2} \ \text{is satisfied}.
	\end{aligned}
	\]
	
	Next, we construct the inverse mapping of $\psi$.
	For any
	\(\lambda\in H_{\langle 1,u_1,\dots,u_r\rangle}\), since $\lambda,\lambda u_i\in \left\langle 1,u_1,\dots,u_r\right\rangle $, they can be written uniquely in the form
	\begin{equation}\label{aaaa}
	\lambda =\beta+\sum_{j=1}^{r}a_ju_j,\
	\lambda u_i=\alpha_i+\sum_{j=1}^{r}c_{ij}u_j\ \text{for some}\ \beta,a_j,\alpha_i,c_{ij}\in\GF(q).
	\end{equation}
	Now define
	\begin{equation}\label{aaa}
	D=
	\begin{pmatrix}
		-c_{11}+a_1u_1 & \cdots & -c_{1r}+a_ru_1\\
		\vdots & \ddots & \vdots\\
		-c_{r1}+a_1u_r & \cdots & -c_{rr}+a_ru_r
	\end{pmatrix}.
	\end{equation}
	Then by using \eqref{aaaa} we get
	\begin{equation}\label{gU3}
		\begin{aligned}
				D(\sum_{i=1}^r u_iY_i)
			&=\sum_{i=1}^{r}\left(\sum_{j=1}^{r}(-c_{ij}+a_ju_i)u_j\right)Y_i\\
			&=\sum_{i=1}^{r}\left(-\sum_{j=1}^{r}c_{ij}u_j+\sum_{j=1}^{r}a_ju_iu_j\right)Y_i\\
			&=\sum_{i=1}^{r}(\alpha_i-\beta u_i)Y_i
			\in T'(u_1,\dots,u_r).
		\end{aligned}
	\end{equation}
	Obviously, $D(Y_i)\in T'(u_1,\dots,u_r)$ for all $1\leq i\leq k-1$. To prove that $D\in G'$,
	it remains to prove that \(\det(D)\neq 0\).
	
	Let
	\[
	\boldsymbol v_j'=(a_j,-c_{1j},\dots,-c_{rj})
	\ (1\le j\le r).
	\]
	From Corollary \ref{cla}, we know that \(\det(D)\neq 0\) if and only if \(\boldsymbol v_1',\dots,\boldsymbol v_r'\) are linearly independent over \(\GF(q)\).
	
	Now set
	\[
\boldsymbol	v_0=(\beta,\alpha_1,\dots,\alpha_r),\
\boldsymbol	v_j=(a_j,c_{1j},\dots,c_{rj})\in\GF(q)^{r+1}\quad (1\le j\le r).
	\]
	Then, by \eqref{aaaa}, we get
	\[
	\lambda(1,u_1,\dots,u_r)=(1,u_1,\dots,u_r)\begin{pmatrix}
	\boldsymbol	v_0\\
	\boldsymbol	v_1\\
		\vdots\\
	\boldsymbol	v_r
	\end{pmatrix}.
	\]
	Since \(\lambda\neq 0\), the multiplication by \(\lambda\) is an invertible \(\GF(q)\)-linear map on
	\(\langle 1,u_1,\dots,u_r\rangle\). Hence
	\(\boldsymbol v_1,\dots,\boldsymbol v_r\) are linearly independent over \(\GF(q)\).
	On the other hand,
	\[
\boldsymbol	v_j'=\boldsymbol v_j \operatorname{diag}(1,-1,\cdots,-1),
	\]
	so \(\boldsymbol v_1',\dots,\boldsymbol v_r'\) are also linearly independent over \(\GF(q)\). Thus \(\det(D)\neq 0\) and $D\in G'$.
	Hence we have a well-defined mapping
	\[
	\begin{aligned}
		\chi:H_{\langle 1,u_1,\dots,u_r\rangle}
		&\longrightarrow G'\\
		\lambda&\longmapsto D,
	\end{aligned}
	\]
	where $\lambda$ satisfies \eqref{aaaa} and $D$ is of form \eqref{aaa}.
	
	Finally, we prove that $\chi$ is the inverse mapping of $\psi$.  For any \(D\in G'\), let
	\[
	D=\begin{pmatrix}
		a_{11}+b_1u_1   & \cdots & a_{1r}+b_ru_1\\
		\vdots        & \ddots & \vdots\\
		a_{r1}+b_1u_r & \cdots & a_{rr}+b_ru_r
	\end{pmatrix},
	\] and $D(\sum_{i=1}^r u_iY_i)
	=\sum_{i=1}^{r}(\alpha_i+\beta u_i)Y_i,$ where $\alpha_i,\beta\in\GF(q)$.
	Then $\psi(D)=\lambda=\sum_{j=1}^{r}b_ju_j-\beta$ and by \eqref{xxx}, we have
	\[
	\lambda u_i=\alpha_i-\sum_{j=1}^{r}a_{ij}u_j\ \text{for}\ i=1,\dots,r.
	\]
	Hence
	\[
	\chi\psi(D)=\begin{pmatrix}
		a_{11}+b_1u_1 & \cdots & a_{1r}+b_ru_1\\
		\vdots & \ddots & \vdots\\
		a_{r1}+b_1u_r & \cdots & a_{rr}+b_ru_r
	\end{pmatrix}=D.
	\]
	Conversely, for any \(\lambda\in H_{\langle 1,u_1,\dots,u_r\rangle}\), let $\lambda=\sum_{j=1}^{r}a_ju_j+\beta$, and for each \(i=1,\dots,r\), $\lambda u_i=\alpha_i+\sum_{j=1}^{r}c_{ij}u_j,$ where $a_i,\beta,\alpha_i,c_{ij}$ are uniquely determined in $\GF(q)$.
	Then
	\[
	\chi(\lambda)=\begin{pmatrix}
		-c_{11}+a_1u_1   & \cdots & -c_{1r}+a_ru_1\\
		\vdots        & \ddots & \vdots\\
		-c_{r1}+a_1u_r & \cdots & -c_{rr}+a_ru_r
	\end{pmatrix}.
	\]
	From \eqref{gU3}, we have
	$$
	\chi(\lambda)(\sum_{i=1}^{r}u_iY_i)=\sum_{i=1}^{r}(\alpha_i-\beta u_i)Y_i.
	$$
	Hence $\psi\chi(\lambda)=\sum_{j=1}^{r}a_ju_j+\beta=\lambda$.
	
Therefore \(\psi\) and \(\chi\) are mutually inverse mappings, and hence the claim is true. Thus $|G'|=\left|H_{\langle 1,u_1,\dots,u_r\rangle}\right|$ and the conclusion follows by \eqref{Gs}.

	\section{The incidence matrix $M_{2,k}$  }\label{section4}
	In this section, we explicitly determine the entries in all $\mathcal F_{k,r}^{k-1}$-column blocks of the $G$-incidence matrix $M_{2,k}$, meaning that the corresponding columns are indexed by $G$-orbit representatives in $\mathcal F_{k,r}^{k-1}$, where $1\leq r\leq k-1$. This will yield many new infinite families of \(q\)-analogs of group divisible designs with arbitrary block dimension.
	
	We first present the known representatives of $G$-orbits on $\Omega_{k}^1$ and $\Omega_{k}^k$, together with the
	$(\mathcal F_2^1,\mathcal F_k^1)$-, $(\mathcal F_2^1,\mathcal F_k^k)$-, $(\mathcal F_2^2,\mathcal F_k^1)$- and $(\mathcal F_2^2,\mathcal F_k^k)$-submatrices of $M_{2,k}$, where $k\geq 2$. Recall that $\mathcal O_k$ is a complete set of representatives for the $H$-orbits on $\left[  \GF(q)^l \atop k \right] $.
	\begin{lem}\rm\cite{WZ}\label{fi1}
		For each integer \(k\geq 2\), we have the following:
		\begin{enumerate}
			\item [\rm(i)] A complete set of representatives of $G$-orbits on $\Omega_{k}^1$ is
			\[
			\mathcal{F}_k^1=\left\{WY_1:W\in\mathcal{O}_k\right\}.
			\]
			\item [\rm(ii)] For \( k\le m\), the action of \(G\) on \(\Omega_k^k\) gives a single orbit $W^G$, where $W=\left\langle Y_1,\dots,Y_k \right\rangle $. In brief,
			\[
			\mathcal{F}_k^k=\left\{\langle Y_1,\dots,Y_k\rangle\right\}.
			\]
			\item [\rm(iii)] For $k\leq l$, the $(\mathcal F_2^1,\mathcal F_k^1)$-submatrix of $M_{2,k}$ coincides with the $H$-incidence matrix between $\left[\GF(q)^l\atop 2 \right] $ and $\left[\GF(q)^l\atop k \right] $.
			\item [\rm(iv)] For $l\geq 2$ and $3\leq k\leq m$, the $(\mathcal F_2^1,\mathcal F_k^k)$-submatrix of $M_{2,k}$ is a zero matrix.
			\item [\rm(v)] For $m\geq 2$ and $3\leq k\leq l$, the $(\mathcal F_2^2,\mathcal F_k^1)$-submatrix of $M_{2,k}$ is a zero matrix.
			\item [\rm(vi)] For $l\geq 2$ and $3\leq k\leq m$, the $(\mathcal F_2^2,\mathcal F_k^k)$-submatrix $R$ of $M_{2,k}$  has exactly one element, that is
			\[R= \left( q^{(l-1)\left(\left( k\atop 2\right) -1\right)}
			\prod_{i=2}^{k-1}\frac{q^{(m-i)l}-1}{q^{\,k-i}-1}\right) .\]
		\end{enumerate}
	\end{lem}

	\subsection{$\left( \mathcal{F}_2^1,\mathcal{F}_{k,r}^{k-1} \right)$-submatrix}
	\begin{lem}\label{valueOkur1}

     For $r=1$, the $\left( \mathcal{F}_2^1,\mathcal{F}_{k,1}^{k-1} \right)$-submatrix of $M_{2,k}$ is a diagonal matrix
    $E=\operatorname{diag}(e,\dots,e)$ of order $n_2^l$, where
   $e=\frac{\prod_{i=1}^{k-2}(q^{ml}-q^{il})}{\prod_{i=2}^{k-1}(q^{k}-q^{i})}.$

	\end{lem}
	\begin{proof}
	By Lemma \ref{fi1}, take $WY_1\in\mathcal F_2^1$ and $T(v)=\left\langle Y_1,\dots,Y_{k-1},vY_1\right\rangle \in \mathcal F_{k,1}^{k-1}$, where \(W=\langle u_1,u_2\rangle\in\left[ \GF(q)^l \atop 2\right] \),
		\(\langle 1,v\rangle\in\left[ \GF(q)^l \atop 2\right]\).
		Then
		\[\Omega:=\{\,g(T(v)):\ WY_1\subseteq g(T(v)),\ g\in G\,\}=\{\,h^{-1}(T(v)):\ h\in H'\,\},\] where \[
		H':=\{\,h\in G:\ h(u_1Y_1),\,h(u_2Y_1)\in T(v)\,\}.
		\]
		Assume that $h\in H'$ and
		\[
		h(Y_1)=z_1Y_1+\cdots+z_{k-1}Y_{k-1}, z_i\in \GF(q^l).
		\]
		 It follows that
		\[
		\bigl(h(u_1Y_1), h(u_2Y_1)\bigr)
		=
		(Y_1,\dots, Y_{k-1})
		\begin{pmatrix}
			z_1u_1 & z_1u_2\\
			\vdots & \vdots\\
			z_{k-1}u_1 & z_{k-1}u_2
		\end{pmatrix}.
		\]
		Since \(h(u_1Y_1),h(u_2Y_1)\in T(v)\), it follows from
		the definition of $T(v)$ that $z_1u_1,z_1u_2\in\left\langle 1,v\right\rangle $ and \(z_iu_1,z_iu_2\in\GF(q)\) for \(2\le i\le k-1\). Furthermore, we have $z_1\left\langle u_1,u_2\right\rangle =\left\langle 1,v \right\rangle $ and $z_i= 0$ for \(2\le i\le k-1\). Hence, $H'\neq\emptyset$ if and only if $\left\langle u_1,u_2\right\rangle $ and $\left\langle 1,v\right\rangle $ lie in the same $H$-orbit. From Lemmas \ref{mapOur} and \ref{fi1}, we get that the $\left( \mathcal{F}_2^1,\mathcal{F}_{k,1}^{k-1} \right)$-submatrix of $M_{2,k}$ is a diagonal matrix by arranging the rows and columns appropriately. If $H'\neq\emptyset$, then $h\in H'$ has the form
		\[\begin{pmatrix}
			z_1 & \boldsymbol{b}\\
			\boldsymbol{0} & A
		\end{pmatrix},
		\] where $z_1\left\langle u_1,u_2\right\rangle=\left\langle1,v \right\rangle  $, $\boldsymbol{b}\in M_{1\times m}(q^l)$ and $A\in \GL(m-1,q^l)$. Obviously,
		$$\left| H'\right| =\left| H_{\left\langle u_1,u_2\right\rangle }\right| \prod_{i=1}^{m-1} \left(q^{ml} - q^{il}\right).$$
		
		Now, we can define an equivalence relation on $H'$ by $h_1\sim h_2$ if and only if $h_1^{-1}(T(v))=h_2^{-1}(T(v))$. Then the equivalence classes are in one-to-one correspondence with the members of $\Omega$. Moreover, $h_1\sim h_2$ if and only if $h_2 h_1^{-1} \in G_{T(v)}$, so each equivalence class is a left coset of $G_{T(v)}$ in $H'$, and hence has size $\left|G_{T(v)} \right| $. Therefore, from Lemma \ref{G_Our}, we get that the $(WY_1,T(v))$-entry of $M_{2,k}$ equals
		$$|\Omega|=|H^{\prime}| /|G_{T(v)}|=\frac{\prod_{i=1}^{k-2}(q^{ml}-q^{il})}{\prod_{i=2}^{k-1}(q^{k}-q^{i})}=e.$$
	\end{proof}
	
\begin{lem}\label{21kr}
		For $2\leq r\leq k-1$, the  $\left( \mathcal{F}_2^1,\mathcal{F}_{k,r}^{k-1} \right)$-submatrix of $M_{2,k}$  is a zero matrix.
\end{lem}

	\begin{proof}
		From \eqref{Tur}, let $v_j=\sum_{i=1}^{r}(a_{ij}+b_ju_i)Y_i+\sum_{i=r+1}^{k-1}a_{ij}Y_i$, for $j=1,2$, be any two vectors of $T(u_1,\dots,u_r)$ such that $\dim_{\GF(q^l)}(\left\langle v_1,v_2\right\rangle_{\GF(q^l)} )=1.$ By Lemma \ref{vvv}, we have that $v_1,v_2$ are linearly dependent over $\GF(q)$.
		
		Now, for a contradiction, let $W:=\left\langle v_1,v_2\right\rangle Y_1\in \mathcal F_2^1 $ and assume that there exist $g\in G$ and $T:=T(u_1,\dots,u_r)\in\mathcal F_{k,r}^{k-1}$ such that $W\subseteq g(T')$, that is, $g^{-1}(W)=\left\langle v_1,v_2\right\rangle g^{-1}(Y_1)\subseteq T $. By the first part of the proof, $\left\langle v_1,v_2\right\rangle g^{-1}(Y_1)$ is a $1$-subspace over $\GF(q)$, thus contradicting the fact that $v_1,v_2$ are linearly independent over $\GF(q)$. Therefore,
		the \(\bigl(\mathcal{F}_2^1,\mathcal{F}_{k,r}^{k-1}\bigr)\)-submatrix is a zero matrix.
	\end{proof}

\subsection{\((\mathcal{F}_2^2,\mathcal{F}_{k,r}^{k-1})\)-submatrix}
	\begin{lem}\label{valueOur2}
	\begin{enumerate}
		\item [\rm(i)]For $r=1$, the \((\mathcal{F}_2^2,\mathcal{F}_{k,1}^{k-1})\)-submatrix of $M_{2,k}$ consists of exactly one row \(P\) of length $n_2^l$, namely,
		\[P=\begin{cases}
		\left( a,\dots,a\right),  &\textnormal{if}\ 2\nmid l,                  \\
		\left(a,\dots,a,b \right), &\textnormal{if}\ 2\mid l,
		\end{cases} \]
		where \begin{equation*}
			\begin{aligned}
				a=\frac{\Big((q^k-1)(q^k-q)-(q^2-1)(q^2-q)\Big)
					\prod_{i=2}^{k-2}(q^{ml}-q^{il})}{(q-1)
					\prod_{i=2}^{k-1}(q^k-q^i)},\\
				b=\frac{\Big((q^k-1)(q^k-q)-(q^2-1)(q^2-q)\Big)
					\prod_{i=2}^{k-2}(q^{ml}-q^{il})}{(q^2-1)
					\prod_{i=2}^{k-1}(q^k-q^i)}.
			\end{aligned}
		\end{equation*}

		\item[\rm(ii)]For \(2\le r\le k-1\), let $r_1,\dots,r_s$ be all positive integers with $r_i\mid\gcd(r+1,l)$. Then the \((\mathcal{F}_2^2,\mathcal{F}_{k,r}^{k-1})\)-submatrix of $M_{2,k}$ consists of exactly one row \(Q\) of length \(n_{r+1}^l\), namely
		\[
		Q=(\underbrace{b_{r_1},\dots,b_{r_1}}_{ n_{r+1,r_1}^l },\dots,\underbrace{b_{r_s},\dots,b_{r_s}}_{n_{r+1,r_s}^l}),
		\]
		where
		\[
		b_{r_i}=\frac{(q^k-1)(q^k-q)\displaystyle\prod_{j=2}^{k-2}(q^{ml}-q^{jl})}{(q^{r_i}-1)\displaystyle\prod_{j=r+1}^{k-1}(q^k-q^j)}.
		\]
	\end{enumerate}
	
	\end{lem}
	
	\begin{proof}
	Note that there is only one $G$-orbit on $\Omega_2^2$ with a representative $U=\left\langle Y_1,Y_2\right\rangle $ by Lemma \ref{fi1}. Then
	\[
	\Omega:=\left\lbrace g(T(u_1,\dots,u_r)):U\subseteq g(T(u_1,\dots,u_r)), g\in G\right\rbrace =\left\lbrace h^{-1}(T(u_1,\dots,u_r)): h\in H'\right\rbrace,
	\] where
	\[
	H':=\{\,h\in G:\ h(Y_1),\,h(Y_2)\in T(u_1,\dots,u_r)\,\}.\]
	
	By the definition of \(T(u_1,\dots,u_r)\) in \eqref{Tur}, every element \(h\in H'\) has the form
	\[
	h=
	\left(
	\begin{array}{cc}
		\begin{matrix}
			a_{11}+b_1u_1 & a_{12}+b_2u_1\\
			\vdots & \vdots\\
			a_{r1}+b_1u_r & a_{r2}+b_2u_r\\
			a_{r+1,1} & a_{r+1,2}\\
			\vdots & \vdots\\
			a_{k-1,1} & a_{k-1,2}
		\end{matrix}
		&
		B\\
		\boldsymbol{0} & C
	\end{array}
	\right),
	\]
	where \(a_{ij},b_j\in \GF(q)\) for \(1\le i\le k-1\), \(j=1,2\), $B\in M_{(k-1)\times (m-2)}(q^l), C\in M_{(m-k+1)\times (m-2)}(q^l),$
	and \(\det(h)\neq 0\). Specifically, the first two columns are linearly independent over $\GF(q^l)$.
	
	Firstly, for \(r=1\), the number of the first two columns of $h$ that are linearly independent over \(\GF(q)\) is \((q^k-1)(q^k-q)\). However, when $a_{ij}=0$ for $2\leq i\leq k-1$ and $ j=1,2$, the pair $(a_{11}+b_1u_1, a_{12}+b_2u_1)$ that is linearly independent over $\GF(q)$ must be linearly dependent over $\GF(q^l)$. And the number of such pairs of columns in $h$ is \((q^2-1)(q^2-q)\). Therefore,
	\[|H^{\prime}|=\left( (q^k-1)(q^k-q)-(q^2-1)(q^2-q)\right)  \prod_{i=2}^{m-1}(q^{ml}-q^{il}).
	\]
	
	Secondly, for \(2\le r\le k-1\), Lemma \ref{vvv} shows that the first two columns of $h$ are linearly independent over \(\GF(q)\) if and only if they are also linearly independent over \(\GF(q^l)\).
 Hence
	\[
	|H^{\prime}|=(q^k-1)(q^k-q) \prod_{i=2}^{m-1}(q^{ml}-q^{il}),\ 2\le r\le k-1.
	\]

	Finally, for $1\leq r\leq k-1$, we consider the equivalence relation on $H'$ defined by that $h_1\sim h_2$ if and only if $h_1^{-1}(T(u_1,\dots,u_r))=h_2^{-1}(T(u_1,\dots ,u_r))$. Therefore,
	$$|\Omega|=|H^{\prime}| /|G_{T(u_1,\dots, u_r)}|$$
	and then the conclusion follows from Theorem \ref{sin} and Lemmas \ref{mapOur} and \ref{G_Our}.
	\end{proof}
	
\subsection{Useful submatrix}

From Lemmas \ref{fi1}--\ref{valueOur2}, we get a submatrix $A_k$ of $M_{2,k}$ as follows:
\begin{equation}\label{matrixA_r}
	A_k = \bordermatrix{%
		&\mathcal F_k^1 &\mathcal{F}_{k,1}^{k-1}	& \mathcal{F}_{k,2}^{k-1} &\cdots & \mathcal{F}_{k,k-1}^{k-1} 	& \mathcal{F}_{k}^{k}       \cr
		\mathcal{F}_2^1 &\widetilde{A_k} &E   & \boldsymbol{0}  &\cdots & \boldsymbol{0}   & \boldsymbol{0}      \cr
		\mathcal{F}_2^2 &  \boldsymbol{0} &P  & Q_2 &\cdots &Q_{k-1}     &R        }.
\end{equation}
The matrix $A_k$ has the following properties:
\begin{enumerate}
	\item [\rm(i)] The $(\mathcal F_2^1,\mathcal F_k^1)$-submatrix $\widetilde{A_k}$ coincides with the $H$-incidence matrix between $\left[\GF(q)^l\atop 2 \right] $ and $\left[\GF(q)^l\atop k \right] $;
 	\item [\rm(ii)]
	$E=\operatorname{diag}(e,\dots,e)$ is a diagonal matrix of order $n_2^l$, where
	$e=\frac{\prod_{i=1}^{k-2}(q^{ml}-q^{il})}{\prod_{i=2}^{k-1}(q^{k}-q^{i})};$
	\item [\rm(iii)] \(P\) is one-row matrix of length $n_2^l$, namely,
	\[P=
	\begin{cases}
	\left( a,\dots,a\right),  	&\textnormal{if}\ 2\nmid l,                  \\
		\left(a,\dots,a,b \right), &\textnormal{if}\ 2\mid l,
	\end{cases} \]
	where
	\begin{equation*}
		\begin{aligned}
			a=\frac{\Big((q^k-1)(q^k-q)-(q^2-1)(q^2-q)\Big)
				\prod_{i=2}^{k-2}(q^{ml}-q^{il})}{(q-1)
				\prod_{i=2}^{k-1}(q^k-q^i)},\\
			b=\frac{\Big((q^k-1)(q^k-q)-(q^2-1)(q^2-q)\Big)
				\prod_{i=2}^{k-2}(q^{ml}-q^{il})}{(q^2-1)
				\prod_{i=2}^{k-1}(q^k-q^i)};
		\end{aligned}
	\end{equation*}
	\item [\rm(iv)]
	For $2\leq r\leq k-1$, \(Q_r\) is one-row matrix of length \(n_{r+1}^l\), namely
	\[
	Q_r=(\underbrace{b_{r_1},\dots,b_{r_1}}_{ n_{r+1,r_1}^l },\dots,\underbrace{b_{r_s},\dots,b_{r_s}}_{n_{r+1,r_s}^l}),
	\]
	where
	\[
	b_{r_i}=\frac{(q^k-1)(q^k-q)\displaystyle\prod_{j=2}^{k-2}(q^{ml}-q^{jl})}{(q^{r_i}-1)\displaystyle\prod_{j=r+1}^{k-1}(q^k-q^j)};
	\]
	\item [\rm(v)]
	$R =  \bigg(q^{(l-1)(\binom{k}{2}-1)}\prod_{i=2}^{k-1}\frac{q^{(m-i)l}-1}{q^{k-i}-1} \bigg).$
\end{enumerate}

\section{Constructions of $q$-GDDs, $q$-PBDs and subspace $2$-designs }\label{section5}
In this section, we are concerned with constructions of infinite families of \(q\)-analogs of group divisible designs and their applications. In Subsection \ref{section51}, we exploit the matrix $A_k$ in \eqref{matrixA_r} to construct new infinite families of simple $q$-GDDs. In Subsection \ref{section52}, we establish a recursive construction for $q$-PBDs and obtain several new simple $q$-PBDs by using known $q$-PBDs of small parameters. In Subsection \ref{section53}, we construct many new infinite families of non-simple subspace $2$-designs by applying the ``Filling in holes" construction in Theorem \ref{aaabb} to the newly constructed $q$-GDDs. We also consider breaking up blocks in $q$-PBDs.

\subsection{Simple $q$-GDDs}\label{section51}
We recall the definition of a Desarguesian \(l\)-spread.
Let \(V\) be a \(v\)-dimensional vector space over \(\GF(q)\) and assume that \(l\mid v\).
An \emph{\(l\)-spread} of \(V\) is a \(1\text{-}(v,l,1)_q\) design, that is, a collection
\(\{W_1,\dots,W_s\}\) of \(l\)-subspaces of \(V\) such that
\(W_i\cap W_j=\{0\}\) for \(1\le i<j\le s\) and \(\bigcup_{i=1}^s W_i = V\).
 It follows that $s=\frac{q^{v}-1}{q^{l}-1}.$
It is well known that an \(l\)-spread of \(V\) exists if and only if \(l\mid v\).

As before, we identify the $ml$-dimensional space over $\GF(q)$ with $\GF(q^{ml})$, also with an $m$-dimensional space over $\GF(q^l)$.
Then the set of \(1\)-dimensional \(\GF(q^l)\)-subspaces of \(\GF(q^l)^m\), regarded as
\(l\)-subspaces of \(V\), forms an \(l\)-spread of \(V\). This spread
is called the \emph{Desarguesian \(l\)-spread} of \(V\). In the sequel, it will serve as the
group set in our \((ml,l,k,\lambda)_q\)-GDD constructions.

As shown
in~\cite{buratti} and \cite{WZ}, a \(q\)-GDD with $\Omega_k^k$ as block set has the Desarguesian \(l\)-spread in \(V\) as its group set,
and admits $\GL(m,q^l)$ as an automorphism group, see Theorem \ref{kkkk} for the parameters. We may produce $q$-GDDs with a much wider range of parameters.

	\begin{thm}\label{GDD}
	Let $m,l,k$ be positive integers satisfying $3\le k\le \min\{m+1,l\}$. Then there exists a simple \((ml,l,k,\lambda)_q\)-GDD on \(\GF(q)^{ml}\) with Desarguesian \(l\)-spread \(\mathcal G\) as group set and admitting
		\(\GL(m,q^l)\) as an automorphism group, where
		\[
			\lambda
			=
			\sum_{r=2}^{k-1}\sum_{u\mid\gcd(r+1,l)}
			w_{r,u}\,\frac{(q^k-1)(q^k-q)\prod_{j=2}^{k-2}(q^{ml}-q^{jl})}
			{(q^{u}-1)\prod_{j=r+1}^{k-1}(q^k-q^j)}\\
			+
			w\,q^{(l-1)\left(\binom{k}{2}-1\right)}
			\prod_{j=2}^{k-1}\frac{q^{(m-j)l}-1}{q^{k-j}-1},
		\]
		with $0\leq w_{r,u}\leq n_{r+1,u}^l$ and \(w\in\{0,1\}\).
	\end{thm}

	\begin{proof}
		It is not hard to see that the $2$-subspaces covered by the elements of $\mathcal G$ are exactly the $2$-subspaces in $\Omega_2^1$. Consider the matrix \(A_k\) in \eqref{matrixA_r}, for each \(2\le r\le k-1\) and \(u\mid \gcd(r+1,l)\), we may select \(w_{r,u}\) distinct members of \(\mathcal{F}_{k,r}^{k-1}\), denoted by $
		K_{r,u}^{1},\dots,K_{r,u}^{w_{r,u}},$
		so that, for every \(U\in\mathcal F_2^2\) and every \(1\le t\le w_{r,u}\), the \((U,K_{r,u}^{t})\)-entry is
		$
		\frac{(q^k-1)(q^k-q)\prod_{j=2}^{k-2}(q^{ml}-q^{jl})}
		{(q^{u}-1)\prod_{j=r+1}^{k-1}(q^k-q^j)}.$
		 Also, if $w=1$, we choose $W\in\mathcal{F}_k^k$. Define
		\[
		\mathcal{B}=\left( \bigcup_{r=2}^{k-1}\bigcup_{u\mid\gcd(r+1,l)} \left( \left( K_{r,u}^{1}\right) ^G\cup\cdots\cup \left( K_{r,u}^{w_{r,u}}\right) ^G\right) \right)\bigcup W^G   .\]
		Then $(V,\mathcal{B})$ is an $(ml,l,k,\lambda)_q$-GDD with $\lambda$ as desired. Because $w_{r,u}\leq n_{r+1,u}^l$, the $q$-GDD is simple.
	\end{proof}
	
Before proceeding further, we record the following immediate consequence of the above theorem. Its special cases will be crucial for our constructions of $q$-PBDs and subspace $2$-designs presented in the remainder of the paper.
	
\begin{cor}\label{3456}
	Let $m,l,k$ be positive integers satisfying $3\le k\le \min\{m+1,l\}$. Then there exist $(ml,l,k,\lambda)_q$-GDDs on \(\GF(q)^{ml}\) with Desarguesian \(l\)-spread \(\mathcal G\) as group set and admitting \(\GL(m,q^l)\) as an automorphism group, where the parameters are the following:
	\begin{enumerate}
		\item[\rm(i)] For \(k=3\), $\lambda
		=
		w_{2,1}\,(q^2+q+1)(q^3-q)$
		with \( w_{2,1}\le\lfloor \frac{(q^{l-1}-1)(q^{l-2}-1)}{(q^3-1)(q^2-1)} \rfloor;\)
		
		\item[\rm(ii)] For $k=4,5,6$,
		$\lambda
		=
		\frac{(q^k-1)(q^k-q)\prod_{j=2}^{k-2}(q^{ml}-q^{jl})}
		{(q^{\gcd(3,l)}-1)\prod_{j=3}^{k-1}(q^k-q^j)}.$
	\end{enumerate}
\end{cor}

\begin{proof}
	Apply Theorem~\ref{GDD} with $w=0$. For $k=3$, take $1\leq w_{2,1}\leq n_{3,1}^l$ and $w_{2,u}=0$ if $u\neq 1$. Then we have the result in $\rm(i)$ by noting that $n_{3,1}^l=\frac{q-1}{q^l-1}
	\sum_{ t\mid \gcd(3,l)}
	\mu\!\left(t\right)
	\left[ l/t\atop 3/t\right]_{q^t}=\lfloor \frac{(q^{l-1}-1)(q^{l-2}-1)}{(q^3-1)(q^2-1)} \rfloor.$ For $k=4,5,6$, take $w_{r,u}=1$ if $(r,u)=(2,\gcd(3,l))$ and $w_{r,u}=0$ otherwise. Then the conclusion in $\rm(ii)$ follows.
\end{proof}

\begin{rmk}
	Theorem \ref{GDD} yields $q$-GDDs for arbitrary block dimension $k$. To the best of our knowledge, the constructions previously available in \cite{Wz,WZ} were restricted to the special cases $k=3$, $(k,q)=(4,2)$ or $\gcd(k,l)=1$. Therefore, the present method substantially extends the known constructions to a much broader setting.
\end{rmk}

\subsection{Simple $q$-PBDs}\label{section52}

\begin{thm}\label{pbd}
	Let $m,l,k$ be positive integers satisfying $3\le k\le \min\{m+1,l\}$. If there exists a $q$-PBD$(l,K,\lambda)$ design admitting $H$ as an automorphism group such that
	$$\begin{aligned}
		\lambda =\sum_{r=2}^{k-1}\sum_{u\mid\gcd(r+1,l)}
		w_{r,u}\,\frac{(q^k-1)(q^k-q)\prod_{j=2}^{k-2}(q^{ml}-q^{jl})}
		{(q^{u}-1)\prod_{j=r+1}^{k-1}(q^k-q^j)}\\
		+
		w\,q^{(l-1)\left(\binom{k}{2}-1\right)}
		\prod_{j=2}^{k-1}\frac{q^{(m-j)l}-1}{q^{k-j}-1},
	\end{aligned}$$ 	where $0\leq w_{r,u}\leq n_{r+1,u}^l$ and \(w\in\{0,1\}\). Then there exists a $q$-PBD$(ml,K\cup \left\lbrace k \right\rbrace , \lambda )$ design admitting $G$ as an automorphism group. If the original $2$-$(l,K,\lambda)_q$ design is simple, then the resulting $q$-PBD is also simple.
\end{thm}

\begin{proof}
	Suppose that \((\GF(q)^l,\mathcal{D})\) is a \(2\)-\((l,K,\lambda)_q\) design admitting
	\(H\) as an automorphism group.
	Let \(W_1,\dots,W_n\) be a complete set of representatives of the \(H\)-orbits on
	\(\mathcal{D}\). Juxtapose the $G$-incidence matrices $M_k$, $k \in K$, horizontally to obtain a new matrix $M$. By Lemma \ref{fi1} and the properties of $A_k$, we may choose the columns
	corresponding to \(W_1Y_1,\dots,W_nY_1\) from the corresponding blocks
	\(\mathcal{F}_k^1\) (\(k\in K\)) so that the sum of each row in $M$ indexed by every representative
	in \(\mathcal{F}_2^1\) is equal to \(\lambda\).
	
	On the other hand, by Theorem~\ref{GDD}, there exists a simple
	\((ml,l,k,\lambda)_q\)-GDD \((\GF(q)^{ml},\mathcal{G},\mathcal{B})\) with Desarguesian \(l\)-spread \(\mathcal G\) as group set and admitting
	\(\GL(m,q^l)\) as an automorphism group. Observe that the block set $\mathcal B\subseteq \Omega_{k}^{k-1}\cup \Omega_{k}^{k}$.
	Define
	\[
	\mathcal{A}:=\mathcal{B}\bigcup \left( \bigcup_{s=1}^{n}(W_{s}Y_1)^G\right) .
	\]
	For any $2\leq r\leq k-1$, the \((\mathcal{F}_2^1,\mathcal{F}_{k,r}^{k-1})\)-,
	\((\mathcal{F}_2^1,\mathcal{F}_k^k)\)-, and \((\mathcal{F}_2^2,\mathcal{F}_k^1)\)-submatrices
	of \(A_k\) are zeros, it follows from the Kramer–Mesner method that
	\((\GF(q)^{ml},\mathcal{A})\) is a
	$q$-PBD$(ml,K\cup \left\lbrace k \right\rbrace , \lambda )$ admitting
	\(G\) as an automorphism group.
	
	Finally, assume that the original $q$-PBD$(l,K,\lambda)$ $(V,\mathcal D)$ is simple. Since
	$W_1,\dots,W_n$ form a complete set of representatives of the $H$-orbits on $\mathcal D$,
	the $G$-orbits
	\[
	(W_1Y_1)^G,\dots,(W_nY_1)^G
	\]
	are pairwise distinct. Moreover, every block in $\bigcup_{s=1}^n (W_sY_1)^G$ is contained in some group of
	the Desarguesian spread $\mathcal G$, whereas every block of the $q$-GDD
	$(\GF(q)^{ml},\mathcal G,\mathcal B)$ meets each group of $\mathcal G$ in dimension at most $1$.
	Therefore, $\mathcal A$ contains no repeated blocks, and so the resulting
	$q$-PBD is simple.
\end{proof}

\begin{thm}\label{pbd1}
 For $q=2$ and any positive integer $m\geq 2$,	there exists a simple $q$-PBD$(ml,K,42s)$, if one of the following is satisfied:
	\begin{enumerate}
	\item[\rm(i)] $K=\left\lbrace 3,4\right\rbrace $,
	\begin{itemize}
	\item $l=7$ and $s=1,2,3$,
	\item $l=8$ and $1 \le s \le 17$,
	\item $l=9$ and $1 \le s \le 66$, $s\neq 33$;
	\end{itemize}
	\item[\rm(ii)] $K=\left\lbrace 3,4,5\right\rbrace $,
	\begin{itemize}
	\item $l=6$ and $s=1$,
	\item $l=7$ and $s=3,4,5,6$.
	\end{itemize}
	\end{enumerate}
\end{thm}

\begin{proof}
 In each of the stated cases, a simple $(ml,l,3,42s)_q$-GDD admitting $H$ as an automorphism group exists by Corollary~\ref{3456}(i), and $q$-PBD$(l,K,42s)$ design admitting $H$ as an automorphism group exists by \cite[Table~2]{PBD}, where we take the supplementary design if $K=\left\lbrace 3,4\right\rbrace $ and $(l,s)\in\left\lbrace (7,3),(8,9),(8,10),\dots,(8,16),(9,34),(9,35),\dots,(9,66)\right\rbrace $ or $K=\left\lbrace 3,4,5\right\rbrace $ and $(l,s)\in\left\lbrace (6,1),(7,5),(7,6)\right\rbrace $. Therefore, Theorem~\ref{pbd} yields the desired $q$-PBD$(ml,K,42s)$s.
\end{proof}

\subsection{Non-simple subspace $2$-designs}\label{section53}

In this subsection, we make use of $q$-GDDs constructed in Theorem \ref{GDD} to obtain subspace $2$-designs via Theorem \ref{aaabb}. Moreover, if one of the original $(v,m,k,\lambda)_q$-GDD $(V,\mathcal G,\mathcal B)$ and the $2$-design $(Z,\mathcal{D})$ is non-simple, then the resulting
$2$-$(v+n,k,q^{\,n(k-2)}\lambda)_q$ design is also non-simple. We also use a construction by breaking up blocks in a $q$-PBD as follows.

\begin{thm}\label{pbdd}
	Let $(V,\mathcal{B})$ be a $q$-PBD$(v,K,\lambda)$. Suppose that, for each
	$u\in K$, there exists a $2$-$(u,k,\mu)_q$ design. Then there exists a
	$2$-$(v,k,\lambda\mu)_q$ design.
\end{thm}

\begin{proof}
	For each block $B\in\mathcal{B}$ with $\dim(B)=u\in K$, construct a
	$2$-$(u,k,\mu)_q$ design on the $u$-subspace $B$, and denote its block set by
	$\mathcal{C}_B$. Define $\mathcal C$ to be the multiset union by
	\[
	\mathcal{C}:=\bigcup_{B\in\mathcal{B}} \mathcal{C}_B.
	\]
	We claim that $(V,\mathcal{C})$ is a $2$-$(v,k,\lambda\mu)_q$ design.
	
	Let $T$ be any $2$-subspace of $V$. Since $(V,\mathcal{B})$ is a
	$2$-$(v,K,\lambda)_q$ design, the subspace $T$ is contained in exactly
	$\lambda$ blocks of $\mathcal{B}$. For each such block $B$, the input
	$2$-$(\dim(B),k,\mu)_q$ design on $B$ contains $T$ in exactly $\mu$ of its
	$k$-dimensional blocks. Therefore, in the multiset union $\mathcal{C}$, the subspace
	$T$ is contained in exactly $\lambda\mu$ blocks. Hence $(V,\mathcal{C})$ is a
	$2$-$(v,k,\lambda\mu)_q$ design.
\end{proof}

\begin{thm}\label{sub2}
	For $q=2,3,4$ and positive integers $m,k$ with $m\geq k-1$,	there exist non-simple $2$-$(v,k,\lambda)_q$ designs, where the parameters are the following:
	\begin{enumerate}
		\item [\rm(i)] $(7m,3,126)_2$;
		\item [\rm(ii)] $(10m,5,155(2^{10m-6}-2^{14})(2^{10m}-2^{30}))_2$ with $m\equiv 2,3\pmod 4$;
		\item [\rm(iii)] $(10m+1,5,155(2^{10m-3}-2^{17})(2^{10m}-2^{30}))_2$ with $m\equiv 2,3,11,30\pmod {36}$;
		\item[\rm(iv)]$(11m,5,155(2^{11m-6}-2^{16})(2^{11m}-2^{33}))_2$ with $m\equiv 2,3,11,66\pmod {72}$;
		\item[\rm(v)]$(11m+1,5,155(2^{11m-3}-2^{19})(2^{11m}-2^{33}))_2$ with $m\equiv 2,3,83,282\pmod {360}$;
		\item [\rm(vi)] $(12m,5,155(2^{12m-6}-2^{18})(2^{12m}-2^{36})/7)_2$ with $m\equiv 2,3,8,12\pmod {15}$;
		\item [\rm(vii)] $(12m,6,93(2^{12m-11}-2^{13})(2^{12m}-2^{36})(2^{12m}-2^{48})/
		7)_2$ with $m\equiv 2,3,4\pmod 5$;
		\item [\rm(viii)] $(8m,4,520(3^{8m-2}-3^{14}))_3$ with $m\equiv 2\pmod 3$;
		\item [\rm(ix)] $(8m,4,7140(4^{8m-3}-4^{13}))_4$ with $m\equiv 2\pmod 3$;
		\item [\rm(x)] $(8m,4,4836(5^{8m-2}-5^{14}))_5$ with $m\equiv 2\pmod 3$.
	\end{enumerate}
\end{thm}

\begin{proof}
First, note that there exist a trivial (non-simple) $2$-$(3,3,3)_2$ design and a trivial
$2$-$(4,3,3)_2$ design. Applying Theorem~\ref{pbdd} to the $q$-PBD$(7m,\left\lbrace 3,4\right\rbrace,42 )$ given in
{\rm(i)} of Theorem~\ref{pbd1}, we obtain
non-simple $2$-$(7m,3,126)_2$ design.

	Next, we apply Theorem~\ref{aaabb} with $k\in \left\lbrace 4,5,6\right\rbrace $ and $n\in\left\lbrace 0,1\right\rbrace $. Suppose that an $(ml,l,k,\mu_1)_q$-GDD exist. If a $2$-$(l+n,k,\mu_2)_q$ design exists and there is a positive integer $x$ such that $x\mu_2=q^{n(k-2)}\mu_1$, then a $2$-$(ml+n,k,q^{n(k-2)}\mu_1)_q$ design also exist. After computation, we conclude that \(\mu_2\mid q^{n(k-2)}\mu_1\) if \(m\) satisfies the stated conditions. We illustrate the proof in Table \ref{Table1}, showing the sources of subspace $2$-designs that we require. Note that all the required $(ml,l,k,\mu_1)_q$-GDDs exist by Corollary \ref{3456}$\rm(ii)$.
	\begin{table*}[!htp]
		\centering
		\caption{Applications of Theorem \ref{aaabb}}\medskip
		\label{Table1}
		\scalebox{0.85}{
			\begin{tabular}{c|c|c|c|c|c|c|c}
				\hline  $q$&$l$&$k$ & $\mu_1$  & $n$ &$\mu_2$ &$2$-$(l+n,k,\mu_2)_q$ design   & Result \\
				\hline
				2&10&5 & $155(2^{10m-6}-2^{14})(2^{10m}-2^{30})$    & 0 & 765 &{\rm{\cite[Table 1]{tabble1}}} & \rm(ii)\\
				2&10&5 & $155(2^{10m-6}-2^{14})(2^{10m}-2^{30})$   & 1 & 43435 &{\rm{\cite[Table 1]{tabble1}}} & \rm(iii)\\
				2&11&5 & $155(2^{11m-6}-2^{16})(2^{11m}-2^{33})$   & 0 & 43435 &{\rm{\cite[Table 1]{tabble1}}} & \rm(iv)\\
				2&11&5 & $155(2^{11m-6}-2^{16})(2^{11m}-2^{33})$   & 1 & 6347715 &{\rm{\cite[Table 1]{tabble1}}} & \rm(v)\\
				2&12&5 & $155(2^{12m-6}-2^{18})(2^{12m}-2^{36})/7$   & 0 & 6347715 &{\rm{\cite[Table 1]{tabble1}}} & \rm(vi)\\
				2&12&6 & $93(2^{12m-11}-2^{13})(2^{12m}-2^{36})(2^{12m}-2^{48})/7$   & 0 & 2962267 &{\rm{\cite[Table 1]{tabble1}}} & \rm(vii)\\
				3&8&4 & $520(3^{8m-2}-3^{14})$   & 0 & 455 &{\rm{\cite[Table 2]{tabble1}}} & \rm(viii)\\
				4&8&4 & $7140(4^{8m-3}-4^{13})$   & 0 & 5733 &{\rm{\cite[Table 3]{tabble1}}} & \rm(ix)\\
				5&8&4 & $4836(5^{8m-2}-5^{14})$  & 0 & 20181 &{\rm{\cite[Table 4]{tabble1}}} & \rm(x)\\
				 \hline
		\end{tabular} }
	\end{table*}
\end{proof}

\section{Summary and concluding remarks}\label{section6}
	
	This paper focuses on constructions for simple $q$-analogs of group divisible designs and their applications.
	We investigate the \(\GL(m,q^l)\)-incidence matrix between \(2\)-subspaces and \(k\)-subspaces
	of \(\GF(q)^{ml}\). Using this incidence matrix, we
	construct new \(q\)-analogs of group divisible designs. More specifically, the main contributions of this paper are as follows.
	
	\begin{itemize}
		\item We present many infinite families for simple $q$-analogs of group divisible designs with arbitrary block dimension (Theorem \ref{GDD}).
		
		\item We obtain a recursive construction of $q$-PBDs (Theorems \ref{pbd}) and several new infinite families of simple $q$-PBDs (Theorems \ref{pbd1}).
		
		\item We construct many new infinite families of non-simple subspace $2$-designs with block dimension $4$, $5$ and $6$ (Theorem \ref{sub2}).
		
	\end{itemize}
	
The results obtained in this paper naturally lead to several questions for future study.
A more refined analysis of the \(G\)-orbits on \(\Omega_k^j\) $(2\leq j\leq k)$ may produce further recursive
constructions of \(q\)-GDDs and subspace $2$-designs, particularly for larger block dimensions. It is also of interest
to construct additional small ``seed'' designs with suitable parameters, since such designs are
crucial for extending the scope of the recursive approach developed here. Finally, it would be worthwhile to explore whether the new \(q\)-GDDs obtained in this paper can be used to construct further infinite families of \(q\)-PBDs and subspace \(2\)-designs.

\end{document}